\documentclass[11pt]{amsart}%
\usepackage[utf8]{inputenc}
\usepackage{amsmath,amssymb,mathrsfs, amsfonts, amsthm}
\usepackage{amsthm}
\usepackage{comment}
\usepackage{subfig}
\usepackage{graphicx}
\usepackage{todonotes}
\usepackage[margin=1in]{geometry}
\usepackage{color}
\usepackage{amssymb}
\usepackage{amsmath}
\usepackage{amsthm}
\usepackage{url}
\usepackage{framed}
\usepackage{enumitem}
\usepackage{multirow}
\usepackage{setspace}
\usepackage{ulem}%
\usepackage{amsfonts}
\providecommand{\U}[1]{\protect\rule{.1in}{.1in}}
\allowdisplaybreaks[3]
\newtheorem{theorem}{Theorem}

\newtheorem{corollary}[theorem]{Corollary}

\newtheorem{proposition}[theorem]{Proposition}

\begin{document}
\title{Zero set of Zak transform and the HRT Conjecture}
\date{\today}
\author{Vignon Oussa}
\address{Department of Mathematics, Bridgewater State University, Bridgewater, MA 0235, USA}
\email{vignon.oussa@bridgew.edu}

\begin{abstract}
The HRT (Heil-Ramanathan-Topiwala) posits the linear independence of any set of nonzero square-integrable vectors obtained from a single nonzero vector $f$ by applying a finite set of time-frequency shift operators. In this short note, we present findings centered on the zero set of the Zak transform of $f$, and a distinct arrangement involving a finite set of $N$ points on an integer lattice in the time-frequency plane, excluding a specific point.
\end{abstract}
\maketitle


\section{Introduction}

The HRT Conjecture \cite{Heil} posits that a finite system of time-frequency shifted copies of a nonzero square-integrable function will always be linearly independent. For an updated overview of this conjecture, please refer to our latest work \cite{KO}. 

Following up on \cite{KO}, the primary objective of this short note is to present a number of new results for this conjecture. Let's consider a function $f$ that is nonzero in $L^{2}\left(\mathbb{R}^{n}\right)$. Choose an arbitrary natural number $N$. Following this, we construct a finite system of vectors, which we will denote as 
$\mathcal{F}_{\left(  N,1\right)  }$ in the following manner
\[
\left\{  t\mapsto e^{-2\pi i\left\langle
y^{\left(  j\right)  },t\right\rangle }f\left(  t-x^{\left(  j\right)
}\right)  :\left(  x^{\left(  1\right)  },y^{\left(  1\right)  }\right)
,\cdots,\left(  x^{\left(  N\right)  },y^{\left(  N\right)  }\right)  \in%
\mathbb{Z}
^{2n}\right\}\newline  \cup\left\{  t\mapsto e^{-2\pi i\left\langle y,t\right\rangle
}f\left(  t-x\right)  \right\}.
\]
This system includes functions of the form
\[
t\mapsto e^{-2\pi i\left\langle y^{\left(  j\right)  },t\right\rangle
}f\left(  t-x^{\left(  j\right)  }\right)  \in L^{2}\left(
\mathbb{R}
^{n}\right)
\]
where each $\left(  x^{\left(  1\right)  },y^{\left(  1\right)  }\right)
,\cdots,\left(  x^{\left(  N\right)  },y^{\left(  N\right)  }\right)  $ belongs
to the integer full-rank lattice $\mathbb{Z}^{2n}.$ Furthermore, it includes a time-frequency-shifted vector of the form $t\mapsto e^{-2\pi
i\left\langle y,t\right\rangle }f\left(  t-x\right)  \in L^{2}\left(
\mathbb{R}
^{n}\right)  $ where $\left(  x,y\right)  $ can be identified with a point on
the time-frequency plane.  We establish the following as valid:

\begin{proposition}
\label{main copy(1)} Suppose that the Zak transform of $f$ is continuous. Then the system of vectors $\mathcal{F}_{\left(  N,1\right)  }$
is linearly independent under the condition that the group which is
generated by $\left(  -x,y\right)  $ modulo $%
\mathbb{Z}
^{2n}$ is a dense subset of the $2n$-dimensional torus $\left[  0,1\right)
^{n}\times\left[  0,1\right)  ^{n}$.
\end{proposition}
The subsequent finding indicates that for functions that decay at a sufficiently rapid rate, the zero set of their Zak transform remains unchanged under a specific symmetry. 
\begin{proposition}
\label{main copy(3)} If $\mathcal{F}_{\left(  N,1\right)  }$ is linearly
dependent then the zero set of the Zak transform of $f$  is invariant under the
action
\[
\left[  0,1\right)  ^{n}\times\left[  0,1\right)  ^{n}\backepsilon z\mapsto
z+\left(  -x,y\right)  \text{ modulo }%
\mathbb{Z}
^{2n}\in\left[  0,1\right)  ^{n}\times\left[  0,1\right)  ^{n}.
\]
\ 
\end{proposition}

\begin{proposition}
\label{main} Suppose that the Zak transform of $f$ is continuous and has a \textcolor{blue}{finite number of zeros} in the unit square $\left[ 0,1\right) ^{n}\times\left[ 0,1\right) ^{n}$. The system of vectors $\mathcal{F}_{\left(  N,1\right)  }$
is linearly independent under the condition that the group which is
generated by $\left(  -x,y\right)  $ modulo $%
\mathbb{Z}
^{2n}$ forms a countably infinite subset (not necessarily dense) of the $2n$-dimensional torus
$\left[  0,1\right)  ^{n}\times\left[  0,1\right)  ^{n}$.
\end{proposition}

As a straightforward application of Proposition \ref{main}, we settle the HRT Conjecture with respect to the configuration at hand for a certain class of totally positive functions. 

\begin{corollary}
\label{main copy(2)} Let $f$ be a totally positive nonzero function without
Gaussian factor in its Fourier transform.  Then the system of vectors
\[
\mathcal{F}=\left\{  t\mapsto e^{-2\pi iy^{\left(  j\right)  }t}f\left(
t-x^{\left(  j\right)  }\right)  :x^{\left(  1\right)  },y^{\left(  1\right)
},\cdots,x^{\left(  N\right)  },y^{\left(  N\right)  }\in%
\mathbb{Z}
\right\}  \cup\left\{  t\mapsto e^{-2\pi iyt}f\left(  t-x\right)  \right\}
\]
is linearly independent.
\end{corollary}

\begin{proof}
Let's consider a function, $f,$ which is totally positive and does not have a
Gaussian factor in its Fourier transform. According to \cite{Kloos}, the Zak
transform of $f$ is continuous and admits only one zero on its fundamental domain
of quasi-periodicity. If $\dim_{%
\mathbb{Q}
}\left(
\mathbb{Q}
+%
\mathbb{Q}
x+%
\mathbb{Q}
y\right)  =1$ then we are in the case of rational time-frequency shifts, and in this case, the HRT Conjecture is known to always be true. Suppose next that $\dim_{%
\mathbb{Q}
}\left(
\mathbb{Q}
+%
\mathbb{Q}
x+%
\mathbb{Q}
y\right)  >1.$ Then the group generated by $\left(  -x,y\right)  $
modulo $%
\mathbb{Z}
^{2}$ is infinite and the stated corollary is immediate. 
\end{proof}

\section{Proof of Propositions \ref{main copy(1)}, \ref{main copy(3)}, and \ref{main} }

Let $\mathfrak{h}_{n}$ be the ($2n+1$)-dimensional Heisenberg algebra. For a
concrete realization \cite{Corwin}, one may choose an algebra constructed on the basis set
$\left\{  X_{1},\cdots,X_{n},Y_{1},\cdots,Y_{n},Z\right\}  $. This algebra can
be expressed as follows: $\sum_{k=1}^{n}x_{k}X_{k}+\sum_{k=1}^{n}y_{k}%
Y_{k}+zZ$ where%
\[
\sum_{k=1}^{n}x_{k}X_{k}+\sum_{k=1}^{n}y_{k}Y_{k}+zZ=\left[
\begin{array}
[c]{ccccc}%
0 & x_{1} & \cdots & x_{n} & z\\
&  &  &  & y_{1}\\
&  & \ddots &  & \vdots\\
&  &  &  & y_{n}\\
0 &  &  &  & 0
\end{array}
\right]  .
\]
Its associated Lie group denoted as $H_{n}$ is given by:
\[
H_{n}=\left\{  \exp\left(  zZ\right)  \exp\left(  \sum_{k=1}^{n}y_{k}%
Y_{k}\right)  \exp\left(  \sum_{k=1}^{n}x_{k}X_{k}\right)  =\left[
\begin{array}
[c]{ccccc}%
1 & x_{1} & \cdots & x_{n} & z\\
&  &  &  & y_{1}\\
&  & \ddots &  & \vdots\\
&  &  &  & y_{n}\\
0 &  &  &  & 1
\end{array}
\right]  :x_{j},y_{k},z\in%
\mathbb{R}
\right\}  .
\]
It is a well-established fact that $H_{n}$ acts unitarily and irreducibly on
$L^{2}\left(
\mathbb{R}
^{n}\right)  $ via $\pi$ as follows: Given points $y=\left(  y_{1}%
,\cdots,y_{n}\right)  ,x=\left(  x_{1},\cdots,x_{n}\right) \in\mathbb{R}^n$ and a vector
$f\in L^{2}\left(
\mathbb{R}
^{n}\right)  ,$ one has: (a) multiplication by characters%
\[
\left[  \pi\exp\left(  zZ\right)  f\right]  \left(  \exp\left(  \sum_{k=1}%
^{n}t_{k}X_{k}\right)  \right)  =e^{2\pi iz}f\left(  \exp\left(  \sum
_{k=1}^{n}t_{k}X_{k}\right)  \right)
\]
(b) a modulation action of the form
\[
\left[  \pi\left(  \exp\left(  \sum_{k=1}^{n}y_{k}Y_{k}\right)  \right)
f\right]  \left(  \exp\left(  \sum_{k=1}^{n}t_{k}X_{k}\right)  \right)
=e^{-2\pi i\left\langle y,t\right\rangle }f\left(  \exp\left(  \sum_{k=1}%
^{n}t_{k}X_{k}\right)  \right)
\]
and (c) a translation action defined as
\[
\left[  \pi\left(  \exp\left(  \sum_{k=1}^{n}x_{k}X_{k}\right)  \right)
f\right]  \left(  t\right)  =f\left(  \exp\left(  \sum_{k=1}^{n}\left(
t_{k}-x_{k}\right)  X_{k}\right)  \right)  .
\]
Suppose that there exists a nonzero vector $f\in L^{2}\left(
\mathbb{R}
^{n}\right)  $ such that
\[
\left\{  \pi\left(  \exp\left(  \sum_{k=1}^{n}\ell_{k}^{\left(  j\right)
}Y_{k}\right)  \exp\left(  \sum_{k=1}^{n}m_{k}^{\left(  j\right)  }%
X_{k}\right)  \right)  f:1\leq j\leq N\right\}  \cup\left\{ \pi\left( \exp\left(
\sum_{k=1}^{n}y_{k}Y_{k}\right)  \exp\left(  \sum_{k=1}^{n}x_{k}X_{k}\right)\right)
f\right\}
\]
is linearly dependent for some finite points lattice points of the form%
\begin{align*}
\sum_{k=1}^{n}\ell_{k}^{\left(  j\right)  }Y_{k}  & \in\sum_{k=1}^{n}%
\mathbb{Z}
Y_{k},\\
\sum_{k=1}^{n}m_{k}^{\left(  j\right)  }X_{k}  & \in\sum_{k=1}^{n}%
\mathbb{Z}
X_{k},1\leq j\leq N
\end{align*}
and
\begin{align*}
\sum_{k=1}^{n}y_{k}Y_{k}  & \in\sum_{k=1}^{n}%
\mathbb{R}
Y_{k},\\
\sum_{k=1}^{n}x_{k}X_{k}  & \in\sum_{k=1}^{n}%
\mathbb{R}
X_{k}.
\end{align*}
In other words, there exist nonzero complex numbers $c_{1},\cdots,c_{N}$ such
that
\[
\sum_{j=1}^{N}c_{j}\pi\left(  \exp\left(  \sum_{k=1}^{n}\ell_{k}^{\left(
j\right)  }Y_{k}\right)  \exp\left(  \sum_{k=1}^{n}m_{k}^{\left(  j\right)
}X_{k}\right)  \right)  f=\pi\left(  \exp\left(  \sum_{k=1}^{n}y_{k}%
Y_{k}\right)  \exp\left(  \sum_{k=1}^{n}x_{k}X_{k}\right)  \right)  f.
\]
Next, let $Z:L^{2}\left(
\mathbb{R}
^{n}\right)  \rightarrow L^{2}\left(  \left[  0,1\right)  ^{n}\times\left[
0,1\right)  ^{n}\right)  $ be the Zak transform \cite{Grog}, formally
defined as follows:
\[
Zf\left(  t,\omega\right)  =\sum_{\tau_{1}\in%
\mathbb{Z}
}\cdots\sum_{\tau_{n}\in%
\mathbb{Z}
}f\left(  \exp\left(  \sum_{k=1}^{n}\left(  t_{k}+\tau_{k}\right)
X_{k}\right)  \right)  e^{-2\pi i\left\langle \omega,\tau\right\rangle }%
\]
where $\tau=\left(  \tau_{1},\cdots,\tau_{n}\right)  .$ Then it is known that
the Zak transform is a unitary operator whose range consists of quasi-periodic
functions (see \cite{Grog}, Section 8.1 on Page 147 and Theorem 8.2.5 on Page
154.) 

Furthermore, it is perhaps worth noting the following properties of the Zak transform:

\begin{itemize}
\item If $f$ is integrable then $Zf\in L^{1}\left(  \left[  0,1\right)
^{n}\times\left[  0,1\right)  ^{n}\right)  $

\item If $f$ is continuous and additionally,
\[
\sum_{\ell\in%
\mathbb{Z}
^{n}}\text{esssup}_{t\in\left[  0,1\right)  ^{n}}\left\vert f\left(
t+\ell\right)  \right\vert <\infty
\]
then $Zf$ is continuous.

\item The Zak transform maps Schwartz functions to smooth function on $\left[
0,1\right)  ^{n}\times\left[  0,1\right)  ^{n}$. Also, if $F$ is quasiperiodic
on $%
\mathbb{R}
^{n}\times%
\mathbb{R}
^{n}$ then $F=Zf$ for some unique Schwartz function $f.$
\end{itemize}

Applying the Zak transform to each side of the following equation:
\[
\sum_{j=1}^{N}c_{j}\pi\left(  \exp\left(  \sum_{k=1}^{n}\ell_{k}^{\left(
j\right)  }Y_{k}\right)  \exp\left(  \sum_{k=1}^{n}m_{k}^{\left(  j\right)
}X_{k}\right)  \right)  f=\pi\left(  \exp\left(  \sum_{k=1}^{n}y_{k}%
Y_{k}\right)  \exp\left(  \sum_{k=1}^{n}x_{k}X_{k}\right)  \right)  f
\]
we obtain:
\begin{align*}
& \sum_{j=1}^{N}c_{j}\sum_{\tau_{1}\in%
\mathbb{Z}
}\cdots\sum_{\tau_{n}\in%
\mathbb{Z}
}e^{-2\pi i\left\langle t+\tau,\ell^{\left(  j\right)  }\right\rangle
}f\left(  \exp\left(  \sum_{k=1}^{n}\left(  t_{k}+\tau_{k}-m_{k}^{\left(
j\right)  }\right)  X_{k}\right)  \right)  e^{-2\pi i\left\langle \omega
,\tau\right\rangle }\\
& =\sum_{\tau_{1}\in%
\mathbb{Z}
}\cdots\sum_{\tau_{n}\in%
\mathbb{Z}
}e^{-2\pi i\left\langle t+\tau,y\right\rangle }f\left(  \exp\left(  \sum
_{k=1}^{n}\left(  t_{k}+\tau_{k}-x_{k}\right)  X_{k}\right)  \right)  e^{-2\pi
i\left\langle \omega,\tau\right\rangle }.
\end{align*}
On the one hand,
\begin{align*}
& \sum_{j=1}^{N}c_{j}\sum_{\tau_{1}\in%
\mathbb{Z}
}\cdots\sum_{\tau_{n}\in%
\mathbb{Z}
}e^{-2\pi i\left\langle t+\tau,\ell^{\left(  j\right)  }\right\rangle
}f\left(  \exp\left(  \sum_{k=1}^{n}\left(  t_{k}+\tau_{k}-m_{k}^{\left(
j\right)  }\right)  X_{k}\right)  \right)  e^{-2\pi i\left\langle \omega
,\tau\right\rangle }\\
& =\sum_{j=1}^{N}c_{j}e^{-2\pi i\left\langle t,\ell^{\left(  j\right)
}\right\rangle }\sum_{\tau_{1}\in%
\mathbb{Z}
}\cdots\sum_{\tau_{n}\in%
\mathbb{Z}
}e^{-2\pi i\left\langle \tau,\ell^{\left(  j\right)  }\right\rangle }f\left(
\exp\left(  \sum_{k=1}^{n}\left(  t_{k}+\tau_{k}-m_{k}^{\left(  j\right)
}\right)  X_{k}\right)  \right)  e^{-2\pi i\left\langle \omega,\tau
\right\rangle }\\
^{\left(  e^{-2\pi i\left\langle \tau,\ell^{\left(  j\right)  }\right\rangle
}=1\right)  }  & =\sum_{j=1}^{N}c_{j}e^{-2\pi i\left\langle t,\ell^{\left(
j\right)  }\right\rangle }\sum_{\tau_{1}\in%
\mathbb{Z}
}\cdots\sum_{\tau_{n}\in%
\mathbb{Z}
}f\left(  \exp\left(  \sum_{k=1}^{n}\left(  t_{k}+\tau_{k}-m_{k}^{\left(
j\right)  }\right)  X_{k}\right)  \right)  e^{-2\pi i\left\langle \omega
,\tau\right\rangle }.
\end{align*}
Moreover, the change of variable $\tau_{k}\mapsto\tau+m_{k}^{\left(  j\right)
}$ yields
\begin{align*}
& \sum_{j=1}^{N}c_{j}\sum_{\tau_{1}\in%
\mathbb{Z}
}\cdots\sum_{\tau_{n}\in%
\mathbb{Z}
}e^{-2\pi i\left\langle t+\tau,\ell^{\left(  j\right)  }\right\rangle
}f\left(  \exp\left(  \sum_{k=1}^{n}\left(  t_{k}+\tau_{k}-m_{k}^{\left(
j\right)  }\right)  X_{k}\right)  \right)  e^{-2\pi i\left\langle \omega
,\tau\right\rangle }\\
& =\sum_{j=1}^{N}c_{j}e^{-2\pi i\left\langle t,\ell^{\left(  j\right)
}\right\rangle }\sum_{\tau_{1}\in%
\mathbb{Z}
}\cdots\sum_{\tau_{n}\in%
\mathbb{Z}
}f\left(  \exp\left(  \sum_{k=1}^{n}\left(  t_{k}+\tau_{k}\right)
X_{k}\right)  \right)  e^{-2\pi i\left\langle \omega,\tau+m^{\left(  j\right)
}\right\rangle }\\
& =\sum_{j=1}^{N}c_{j}e^{-2\pi i\left\langle t,\ell^{\left(  j\right)
}\right\rangle }e^{-2\pi i\left\langle \omega,m^{\left(  j\right)
}\right\rangle }\sum_{\tau_{1}\in%
\mathbb{Z}
}\cdots\sum_{\tau_{n}\in%
\mathbb{Z}
}f\left(  \exp\left(  \sum_{k=1}^{n}\left(  t_{k}+\tau_{k}\right)
X_{k}\right)  \right)  e^{-2\pi i\left\langle \omega,\tau\right\rangle }\\
& =\left(  \sum_{j=1}^{N}c_{j}e^{-2\pi i\left\langle t,\ell^{\left(  j\right)
}\right\rangle }e^{-2\pi i\left\langle \omega,m^{\left(  j\right)
}\right\rangle }\right)  \cdot Zf\left(  t,\omega\right)  .
\end{align*}
Furthermore,
\begin{align*}
& \sum_{\tau_{1}\in%
\mathbb{Z}
}\cdots\sum_{\tau_{n}\in%
\mathbb{Z}
}e^{-2\pi i\left\langle t+\tau,y\right\rangle }f\left(  \exp\left(  \sum
_{k=1}^{n}\left(  t_{k}-x_{k}+\tau_{k}\right)  X_{k}\right)  \right)  e^{-2\pi
i\left\langle \omega,\tau\right\rangle }\\
& =e^{-2\pi i\left\langle t,y\right\rangle }\sum_{\tau_{1}\in%
\mathbb{Z}
}\cdots\sum_{\tau_{n}\in%
\mathbb{Z}
}e^{-2\pi i\left\langle \tau,y\right\rangle }f\left(  \exp\left(  \sum
_{k=1}^{n}\left(  t_{k}-x_{k}+\tau_{k}\right)  X_{k}\right)  \right)  e^{-2\pi
i\left\langle \omega,\tau\right\rangle }\\
& =e^{-2\pi i\left\langle t,y\right\rangle }\sum_{\tau_{1}\in%
\mathbb{Z}
}\cdots\sum_{\tau_{n}\in%
\mathbb{Z}
}f\left(  \exp\left(  \sum_{k=1}^{n}\left(  t_{k}-x_{k}+\tau_{k}\right)
X_{k}\right)  \right)  e^{-2\pi i\left\langle \omega+y,\tau\right\rangle }\\
& =e^{-2\pi i\left\langle t,y\right\rangle }Zf\left(  t-x,\omega+y\right)  .
\end{align*}
In summary, we have
\[
\left(  \sum_{j=1}^{N}c_{j}e^{-2\pi i\left\langle t,\ell^{\left(  j\right)
}\right\rangle }e^{-2\pi i\left\langle \omega,m^{\left(  j\right)
}\right\rangle }\right)  \cdot Zf\left(  t,\omega\right)  =e^{-2\pi
i\left\langle t,y\right\rangle }Zf\left(  t-x,\omega+y\right)  .
\]
For more compact notation, let $\left(  t,\omega\right)  =z\in\left[  0,1\right)  ^{n}\times\left[
0,1\right)  ^{n},F=Zf$, $\gamma=\left(  -x,y\right)  $ and let $P$ be the
trigonometric polynomial given by
\[
P\left(  t,\omega\right)  =\sum_{j=1}^{N}c_{j}e^{-2\pi i\left\langle
t,\ell^{\left(  j\right)  }\right\rangle }e^{-2\pi i\left\langle
\omega,m^{\left(  j\right)  }\right\rangle }.
\]
Given these definitions, we can see that the product of the absolute values of
$P$ and $F$ at the point $z$ equals the absolute value of $F$ at the point $z$
shifted by $\gamma$. In other words, for $z\in\left[  0,1\right)  ^{n}%
\times\left[  0,1\right)  ^{n},$
\[
\left\vert P\left(  z\right)  \right\vert \cdot\left\vert F\left(  z\right)
\right\vert =\left\vert F\left(  z+\gamma\right)  \right\vert .
\]
Furthermore, $\left\vert F\right\vert $ is a $%
\mathbb{Z}
^{2n}$-periodic square-integrable function. 

As referenced in Lemma 8.4.2 of \cite{Grog}, if $F$ is a continuous function,
it is guaranteed to have at least one zero within the unit square defined by
$\left[  0,1\right)  ^{n}\times\left[  0,1\right)  ^{n}.$ By successively
applying the equation
\[
\left\vert P\left(  z\right)  \right\vert \cdot\left\vert F\left(  z\right)
\right\vert =\left\vert F\left(  z+\gamma\right)  \right\vert
\]
we deduce that for any natural number $m,$%
\[
\left\vert F\left(  z+m\gamma\right)  \right\vert =%
{\displaystyle\prod\limits_{j=0}^{m-1}}
\left\vert P\left(  z+j\gamma\right)  \right\vert \cdot\left\vert F\left(
z\right)  \right\vert .
\]

Let $\Gamma$ be the group generated by $\gamma$ modulo $%
\mathbb{Z}
^{2n}.$ Then $\Gamma$ is a countable subgroup of the $2n$-dimensional torus
$\left[  0,1\right)  ^{n}\times\left[  0,1\right)  ^{n}.$

\subsection{Proof of Proposition \ref{main copy(1)}}
If the set $\Gamma$ is dense in the torus $\left[ 0,1\right) ^{n}\times\left[ 0,1\right) ^{n}$, it implies that the zero set of the function $F$ must also be densely distributed in this unit square as well. Considering that $F$ is a continuous function, it logically follows that $F$ must be identically zero across its domain. This validates Proposition \ref{main copy(1)}.

\subsection{Proof of Proposition \ref{main copy(3)}}

Fix $\lambda\in\mathrm{Zero}\left(  Zf\right)  $ then for any natural number
$m,$%
\[
\left\vert F\left(  \lambda+m\gamma\right)  \right\vert =%
{\displaystyle\prod\limits_{j=0}^{m-1}}
\left\vert P\left(  \lambda+j\gamma\right)  \right\vert \cdot\left\vert
F\left(  \lambda\right)  \right\vert =0.
\]
The calculation above suggests that if $\lambda$ belongs to the zero set of $Zf$ (denoted as $\mathrm{Zero}\left( Zf\right)$), then the term $(\lambda + m\gamma) \mod \mathbb{Z}^{2n}$ is also a member of the zero set of $Zf$. This highlights that the zero set of $Zf$ is invariant under the operation of $\Gamma$. More specifically, it underlines that the zero set of $Zf$ remains unaltered under the action of $\gamma$.

\subsection{Proof of Proposition \ref{main}}

Suppose that $\Gamma$ is an infinite set and let $\mathrm{Zero}\left(
Zf\right)  $ be the zero set of the Zak transform of $f.$ Fix $\lambda
\in\mathrm{Zero}\left(  Zf\right)  $ then for any natural number $m,$%
\[
\left\vert F\left(  \lambda+m\gamma\right)  \right\vert =%
{\displaystyle\prod\limits_{j=0}^{m-1}}
\left\vert P\left(  \lambda+j\gamma\right)  \right\vert \cdot\left\vert
F\left(  \lambda\right)  \right\vert =0.
\]
This means that
\[
\lambda\in\mathrm{Zero}\left(  Zf\right)  \Rightarrow\left(  \lambda
+m\gamma\right)  \text{ modulo }%
\mathbb{Z}
^{2n}\in\mathrm{Zero}\left(  Zf\right)
\]
and this shows that $\left\vert \mathrm{Zero}\left(  Zf\right)  \right\vert
=\infty,$ contradicting the assumption that the zero set of the Zak transform of $f$ is finite.

\end{document}